\documentclass[11pt]{article}

\usepackage[a4paper,margin=1in]{geometry}
\usepackage{amsmath,amssymb,amsthm,mathtools}
\usepackage{enumitem}
\usepackage{booktabs}
\usepackage{array}
\usepackage{microtype}
\usepackage[hidelinks]{hyperref}
\usepackage{xcolor}

\usepackage{authblk}

\newtheorem{theorem}{Theorem}[section]
\newtheorem{lemma}[theorem]{Lemma}

\newtheorem{proposition}[theorem]{Proposition}
\newtheorem{corollary}[theorem]{Corollary}
\theoremstyle{remark}

\newcommand{\stirling}[2]{\genfrac{[}{]}{0pt}{}{#1}{#2}}

\newcommand{\Ap}{\operatorname{Ap}}
\newcommand{\lcm}{\operatorname{lcm}}
\newcommand{\Z}{\mathbb Z}

\title{Telescopic Structure of Numerical Semigroups Generated by\\
Unsigned Stirling Numbers of the First Kind}
\author[1,2]{Takao Komatsu}
\author[3,*]{Kyunghwan Song}

\affil[1]{Department of Mathematics, Institute of Science Tokyo,
	Tokyo 152-8551, Japan}

\affil[2]{Jeju Glocal University Project of Jeju National University,
	Jeju National University, 102, Jejudaehak-ro, Jeju-si,
	Jeju-do 63243, Republic of Korea}

\affil[3]{Department of Mathematics, Jeju National University,
	102 Jejudaehakro Jeju, 63243, Republic of Korea}

\affil[*]{Corresponding author: khsong@jejunu.ac.kr}
\date{}

\begin{document}
\maketitle

\begin{abstract}
Let
\[
S_n=\left\langle \stirling n1,\ldots,\stirling n{n-1}\right\rangle
\]
be the numerical semigroup generated by the nontrivial unsigned Stirling
numbers of the first kind in the $n$th row.  Writing
$a=\stirling n{n-1}=\binom n2$ and
$b_j=\stirling n{n-2j}$, we determine the complete gcd filtration of
the reduced generating system.  More precisely, if
\[
M_j=\frac12\operatorname{lcm}\{m\ge1:\varphi(m)\le2j\},
\]
then
\[
\gcd(a,b_1,\ldots,b_j)=\frac{a}{\gcd(a,M_j)}.
\]
The proof uses prime-power block polynomials and exact $p$-adic support
at threshold degrees.  We then establish a divisor-support property for
canonical mixed-radix reductions and combine it with an
elementary-symmetric growth estimate to prove that, after inactive
generators are removed, the resulting generating sequence is
telescopic.  Consequently, the Ap\'ery set of $S_n$ is rectangular,
yielding an explicit Frobenius formula.  The same structure shows that
$S_n$ is free and symmetric and gives explicit formulas for its genus
and conductor, while its type is one.
\end{abstract}



\section{Introduction}\label{sec:Introduction}

Let $A=\{a_1,\ldots,a_k\}$ be a finite set of positive integers with
$\gcd(A)=1$, and let
$$
\langle A\rangle
=
\left\{
\lambda_1a_1+\cdots+\lambda_ka_k:
\lambda_1,\ldots,\lambda_k\in\mathbb Z_{\ge0}
\right\}.
$$
Then $\langle A\rangle$ is a numerical semigroup, and its Frobenius
number is the largest nonnegative integer not belonging to
$\langle A\rangle$.  Although the Frobenius problem is classical, no general analogue of the
two-generator formula is available for three or more generators, and
explicit formulas are typically obtained only for families with additional
structure.  Consequently, an important direction is to identify
structured families of generators for which the arithmetic of the
semigroup can be described explicitly; see, for example,
\cite{Ramirez2005,Rosales2009}.

Combinatorial arrays provide a natural source of such structured
families.  In particular, numerical semigroups generated by entries
coming from binomial-type sequences exhibit an interaction between
the arithmetic of the coefficients and the additive structure of the
resulting semigroup.  Hwang and Song studied Frobenius problems
associated with consecutive binomial coefficients and obtained
explicit formulas in that setting \cite{HwangSong26}.  
Here, we use another standard array, the (unsigned) Stirling numbers of the first kind $\stirling n{k}$. 
They arise from the rising factorial:  
$$
x(x-1)\cdots(x-n+1)=\sum_{k=0}^n(-1)^{n-k}\stirling n{k}x^k. 
$$
They satisfy the recurrence relation 
\begin{equation} 
\stirling{n}{k}=\stirling{n-1}{k-1}+(n-1)\stirling{n-1}{k}
\label{rec:st1}
\end{equation} 
with $\stirling{n}{k}=0$ ($k>n$), $\stirling{n}{0}=0$ ($n>0$) and $\stirling{0}{0}=1$.  
For background on Stirling numbers of the first kind and their generating functions (see, e.g., \cite[Chapter V]{Comtet1974}). 
We consider the set 
$$
A_n=\left\{\stirling n1,\ldots,\stirling n{n-1}\right\}.
$$ 
Since $\stirling n{n}=1$, it should be removed from this set.  
Let $S_n$ be the set of integers that can be expressed as a non-negative integer linear combination of the elements of this set: 
\begin{equation}\label{sn}
S_n=\left\langle
\stirling n1,\stirling n2,\ldots,\stirling n{n-1}
\right\rangle. 
\end{equation} 
We also assume that $n\ge 4$. Otherwise, the result becomes trivial.  
We use standard terminology concerning numerical semigroups, Ap\'ery sets,
and Frobenius numbers as in \cite{Rosales2009}. 
The submonoid $S_n$ is a numerical semigroup precisely when $\gcd(A_n)=1$. When $S_n$ is a numerical semigroup, the set of all non-negative integers excluding the elements of $S_n$, that is, the set of gaps, is a finite set. Its largest element is called the Frobenius number, and in this paper, it is denoted as $F(S_n)$.  That is, for the set of non-negative integers $\mathbb N_0$,   
$$
F(S_n)=\max\bigl(\mathbb N_0\setminus S_n\bigr).
$$
If $\mathfrak b = \min A_n$, the Ap\'ery set with respect to $\mathfrak b$ is
$$
\Ap(S_n;\mathfrak b) = \{m_0,m_1,\ldots,m_{\mathfrak b-1}\},
$$
where $m_i$ is the least member of $S_n$ congruent to
$i\pmod {\mathfrak{b}}$. 
Then, the Frobenius number is obtained from the elements of the Ap\'ery set as 
$$
F(S_n)=\max\{m_i:i=0,1,\dots,\mathfrak b-1\}-\mathfrak{b}. 
$$
We denote by
$$
\mathrm{g}(S_n)=\#(\mathbb N_0\setminus S_n)
$$
the {\it genus} of $S_n$, and by
$$
c(S_n)=F(S_n)+1
$$
its {\it conductor}.  We also use $\mathrm{t}(S_n)$ for the {\it type} of $S_n$, that is the cadinality of the set of all the integer elements $x$ (pseudo-Frobenius numbers) such that $x\not\in S_n$ and $x+s\in S_n$ for all $s\in S_n\backslash\{0\}$.

By Lemmas \ref{lem:min1} and \ref{lem:odd-defect} below, the generating system of the numerical semigroup $S_n$ can be reduced as 
$$
S_n=\langle a,b_1,\ldots,b_{s}\rangle,
$$
where 
$$
a=\stirling n{n-1}=\binom n2,
\qquad
b_j=\stirling n{n-2j},
\qquad
1\le j\le s=\left\lfloor\frac{n-1}{2}\right\rfloor.
$$
The central problem is then to understand the successive gcds
$$
d_j=\gcd(a,b_1,\ldots,b_j)
$$
and the quotients $q_j=d_{j-1}/d_j$.

The first main result gives a complete description of this gcd
filtration.  Define
$$
M_j
=
\frac12\operatorname{lcm}
\{m\ge1:\varphi(m)\le2j\},
\qquad M_0=1,
$$
where $\varphi$ denotes Euler's totient function.  We prove that
$$
d_j=\frac{a}{\gcd(a,M_j)}
$$
for every $1\le j\le s$, and hence
$$
q_j=
\frac{\gcd(a,M_j)}
     {\gcd(a,M_{j-1})}.
$$
Thus the complete gcd filtration of a row of Stirling numbers is
governed by the lcm of the integers whose Euler totients are bounded
by $2j$.  As a first consequence, $a\mid M_s$, so $d_s=1$ and hence
$S_n$ is indeed a numerical semigroup.

The second part of the paper shows that the gcd filtration controls
much more than the final gcd.  For each $j$, we construct a canonical
mixed-radix reduction of $q_jb_j$ modulo the preceding generators.
A key support theorem shows that a nonzero canonical digit at an
index $h<j$ can occur only when $h$ is a proper divisor of $j$.
Combining this divisor-support property with an elementary-symmetric
growth estimate yields
$$
q_jb_j\in
\langle a,b_1,\ldots,b_{j-1}\rangle.
$$
After the inactive indices with $q_j=1$ are deleted, the remaining
generating sequence is therefore telescopic.

This telescopic structure gives a rectangular description of the
Ap\'ery set:
$$
\operatorname{Ap}(S_n,a)
=
\left\{
\sum_{j=1}^{s}x_jb_j:
0\le x_j<q_j
\right\}.
$$
Consequently, the Frobenius number is given by 
$$
F(S_n)
=
\sum_{j=1}^{s}(q_j-1)b_j-a.
$$
Moreover, the resulting numerical semigroup is free and hence
symmetric.  We therefore also obtain that its type is one and $S_n$ is {\it symmetric}, that is, 
$$
\mathrm g(S_n)=\frac{F(S_n)+1}{2}.
$$
Thus the arithmetic filtration not only determines the Frobenius
number but also yields the principal structural invariants of $S_n$.

The proof combines $p$-adic information on the coefficients of
$$
P_n(T)=\prod_{r=1}^{n-1}(1+rT)
$$
with prime-power block polynomials.  Exact support at the threshold
degrees determines the successive gcds, while the squarefreeness of
the quotients $q_j$ forces the divisor-support phenomenon in the
canonical reductions.  This provides the bridge from the arithmetic
of Stirling numbers to the telescopic structure of the associated
numerical semigroups.

The paper is organized as follows.  Section~2 introduces the reduced
generating system and states the main results.  Section~3 develops the
$p$-adic analysis and proves the exact gcd filtration.  Section~4
establishes the divisor-support and positivity results needed for the
canonical reductions.  The final section proves telescopicity and
derives the rectangular Ap\'ery set, the Frobenius formula, and the
remaining structural consequences.

\section{Statement of the main result}\label{sec:main}

The rows of the Stirling numbers of the first kind have strong divisibility and $p$-adic properties, but these properties are considerably less transparent from the viewpoint of additive semigroup generation. 

\begin{lemma}  
For integers $n$ and $m$ with $n\ge 4$ and $1\le m\le n-1$, 
$$
\stirling{n}{n-1}\le\stirling{n}{n-m}\,. 
$$ 
\label{lem:min1}
\end{lemma}  
\begin{proof}
In 1951, Hammersley \cite{Hammersley51} showed that the sequence of $\stirling{n}{k}$ ($k=1,\dots,n$) is unimodal, where the maximizing index $K_n$ is unique (later proved by Erd\H{o}s \cite{Erdos53}) and given by the semi-closed-form formula. Namely, 
$$ 
\stirling{n}{1}<\stirling{n}{2}<\dots<\stirling{n}{K_n-1}\le\stirling{n}{K_n}>\stirling{n}{K_n+1} 
>\dots>\stirling{n}{n-2}>\stirling{n}{n-1}>\stirling{n}{n}=1 
$$ 
with $K_n=K_{n-1}$ or $K_n=K_{n-1}+1$.  
Since 
\begin{equation}
\stirling{n}{n-1}=\binom{n}{2}\quad\hbox{and}\quad \stirling{n}{1}=(n-1)!\,,
\label{eq:stnn11}
\end{equation}
for $n\ge 4$, we have 
$$
\stirling{n}{n-1}\le \stirling{n}{1}
$$
(the equality holds only when $n=4$). 
\end{proof}

Let
$$
P_n(T):=\prod_{r=1}^{n-1}(1+rT)
       =\sum_{\ell=0}^{n-1}E_\ell(n)T^\ell.
$$
Then
$$
E_\ell(n)=e_\ell(1,\ldots,n-1)
         =\stirling n{n-\ell},
$$
which is the standard elementary-symmetric realization of the unsigned
Stirling numbers of the first kind; see, for example,
\cite[Chapter V]{Comtet1974}.
For convenience, we let
$$
a:=\stirling n{n-1}=\binom n2
$$
and
$$
b_j:=E_{2j}(n)=\stirling n{n-2j},
\qquad
1\le j\le s:=\left\lfloor\frac{n-1}{2}\right\rfloor.
$$

We next justify the reduction to the even-defect generators. That is, the generating set in (\ref{sn}) is not minimal. Some elements in $A_n$ can be expressed by other elements of $A_n$.  

\begin{lemma}\label{lem:odd-defect}
For every odd integer $\ell$ with $1\le \ell\le n-1$,
$$
\binom n2=a\mid E_\ell(n)=\stirling n{n-\ell}.
$$
Consequently,
$$
\left\langle
\stirling n1,\ldots,\stirling n{n-1}
\right\rangle
=
\langle a,b_1,\ldots,b_s\rangle.
$$
\end{lemma}

\begin{proof}
Recall
$$
P_n(T)=\prod_{r=1}^{n-1}(1+rT).
$$
The involution $r\mapsto n-r$ permutes $\{1,\ldots,n-1\}$ modulo $n$, and hence
$$
P_n(T)\equiv \prod_{r=1}^{n-1}(1-rT) = P_n(-T) \pmod n.
$$
The factor $1+(n-1)T$ is congruent to $1$ modulo $n-1$, while $r\mapsto n-1-r$ permutes $\{1,\ldots,n-2\}$ modulo $n-1$.  Therefore
$$
P_n(T)\equiv \prod_{r=1}^{n-2}(1-rT) \equiv P_n(-T) \pmod{n-1}.
$$
Since $\gcd(n,n-1)=1$, we have
$$
P_n(T)-P_n(-T)\equiv0\pmod{n(n-1)}.
$$  
For odd $\ell$, the coefficient of $T^\ell$ in this difference is $2E_\ell(n)$.  
Hence $n(n-1)\mid2E_\ell(n),$ or equivalently
$$
\binom n2\mid E_\ell(n).
$$
Thus every odd-defect generator is already a multiple of
$a=\binom n2$, and the asserted reduction follows.
\end{proof}

Now, for the Stirling numerical set $S_n=\langle a,b_1,\ldots,b_s\rangle$, let 
$$
d_0:=a,\quad
d_j:=\gcd(a,b_1,\ldots,b_j),
\quad
q_j:=\frac{d_{j-1}}{d_j}\quad(j=1,2,\dots,s).
$$
Also put
$$
M_j:=\frac12
\lcm\{m\ge1:\varphi(m)\le2j\}\quad(j=1,2,\dots,s),
\qquad M_0:=1.
$$

Our main theorems and corollaries are the following.

\begin{theorem}[Exact gcd filtration]\label{thm:main}
For every $n\ge4$ and
$1\le j\le s=\lfloor(n-1)/2\rfloor$, one has
\[
d_j=\frac{a}{\gcd(a,M_j)}.
\]
Consequently,
\[
q_j=\frac{\gcd(a,M_j)}{\gcd(a,M_{j-1})}.
\]
\end{theorem}

\begin{corollary}\label{cor:numerical}
For every $n\ge4$, $\gcd\left(\stirling n1,\stirling n2,\ldots,\stirling n{n-1}\right)=1.$
Consequently, $S_n$ is a numerical semigroup.
\end{corollary}

\begin{theorem}[Telescopic structure]\label{thm:telescopic}
For every $1\le j\le s$, $q_jb_j\in\langle a,b_1,\ldots,b_{j-1}\rangle.$ Let
$$
J_n=\{j:1\le j\le s,\ q_j>1\}
   =\{j_1<\cdots<j_t\}.
$$
Then $S_n=\langle a,b_{j_1},\ldots,b_{j_t}\rangle,$ and the reduced generating sequence $(a,b_{j_1},\ldots,b_{j_t})$ is telescopic.
\end{theorem}

\begin{corollary}\label{cor:apery-frobenius}
For every $n\ge4$,
$$
\Ap(S_n,a)
=
\left\{
\sum_{j=1}^{s}x_jb_j:
0\le x_j<q_j
\right\}.
$$
In particular, $\max\Ap(S_n,a)=\sum_{j=1}^{s}(q_j-1)b_j,$ and hence $F(S_n)=\sum_{j=1}^{s}(q_j-1)b_j-a.$ Equivalently,
$$
F(S_n)=
\sum_{j=1}^{\lfloor(n-1)/2\rfloor}
(q_j-1)\stirling n{n-2j}
-\binom n2.
$$
\end{corollary}

\begin{corollary}[Structural invariants]\label{cor:structural-invariants}
For every $n\ge4$, the numerical semigroup $S_n$ is free, and hence
symmetric.  In particular,
$$
\mathrm{t}(S_n)=1,
\qquad
\mathrm{g}(S_n)=\frac{F(S_n)+1}{2},
\qquad
c(S_n)=F(S_n)+1.
$$
Consequently,
$$
\mathrm{g}(S_n)
=
\frac12\left(
1-\binom n2+
\sum_{j=1}^{\lfloor(n-1)/2\rfloor}
(q_j-1)\stirling n{n-2j}
\right),
$$
and
$$
c(S_n)
=
1-\binom n2+
\sum_{j=1}^{\lfloor(n-1)/2\rfloor}
(q_j-1)\stirling n{n-2j}.
$$
\end{corollary}

\noindent 
{\bf Example.} 
Let $n=7$.    
Since 
$$
(\varphi(m))_{m=1}^{20}=1, 1, 2, 2, 4, 2, 6, 4, 6, 4, 10, 4, 12, 6, 8, 8, 16, 6, 18, 8,
$$
A direct computation gives
$$
M_1 = 6, \quad M_2 = 60, \quad M_3 = 1260
$$
Hence, for $1\le j\le s=3$ and $a=\binom{7}{2}=21$,  
\begin{align*}
q_1&=\frac{\gcd(21,6)}{\gcd(21,1)}=3,\\
q_2&=\frac{\gcd(21,60)}{\gcd(21,6)}=1,\\
q_3&=\frac{\gcd(21,1260)}{\gcd(21,60)}=7. 
\end{align*}  
Then, the Ap\'ery set is given as 
$$
\Ap(S_7,a)=\left\{\sum_{j=1}^3 x_j\stirling 7{7-2j}:
0\le x_j\le q_j-1\right\}. 
$$ 
Since 
$$
\max\Ap(S_7,a)=2\stirling 7{5}+6\stirling 7{1},  
$$
the Frobenius number is given by 
\begin{align*}
F(S_7)&=2\stirling 7{5}+6\stirling 7{1}-\stirling 7{6}\\
&=2\cdot 175+6\cdot 720-21=4649.  
\end{align*} 
In fact, the Frobenius numbers are calculated as 
\begin{multline*}
\{F(S_n)\}_{n=4}^{20}=49, 121, 1251, 4649, 85471, 261391, 1702147, 36919037, 1205447959, \\
5748032707, 238443327227,  2221478583, 54675258931, 334950599787939,\\
19574499155734199, 115242747197005505, 7768172642857175609. 
\end{multline*} 
Note that $F(S_3)=1$.   
\bigskip

The rest of the paper proves Theorems \ref{thm:main} and \ref{thm:telescopic}, together with their corollaries.

\section{Arithmetic of the gcd filtration}

For a prime $p$ and $\beta\ge1$, let us define $H_{p,\beta}(T):= \prod_{r=1}^{p^\beta-1}(1+rT)$ and $U_{p,\beta}(T):= \prod_{\substack{1\le r<p^\beta\\p\nmid r}}(1+rT).$
Separating the factors divisible by $p$ gives the exact recursion
\begin{equation}\label{eq:H-rec}
H_{p,\beta}(T) = U_{p,\beta}(T)H_{p,\beta-1}(pT).
\end{equation}

In the sequel, we use the following definitions of the valuation for a fixed prime number $p$. $\nu_p \colon \mathbb{N} \rightarrow \mathbb{N} \cup \{0\}$ denotes the $p$-adic valuation for $\mathbb{N}$, which is a function  defined by 
$$
\nu_p(x) = \max\{ \nu \in \mathbb{N} \cup \{0\}: p^{\nu} \mid x\}.
$$
Also, we denote the coefficient of $T^k$ in $F(T)$ as $[T^k]F(T)$.

Firstly, let us recall a basic lemma regarding to the valuation of the binomial coefficient whose upper is the power of prime $p$.

\begin{lemma}\label{lem:binom-valuation}
Let $p$ be prime, $m\ge1$, and $1\le r<p^m$.  Then
$$
\nu_p\left(\binom{p^m}{r}\right)=m-\nu_p(r).
$$
\end{lemma}

\begin{proof}
Use
$$
\binom{p^m}{r}
=
\frac{p^m}{r}\binom{p^m-1}{r-1}.
$$
By Lucas' theorem \cite{Lucas1878},
$$
\binom{p^m-1}{r-1}\not\equiv0\pmod p.
$$
The assertion follows.
\end{proof}

\begin{lemma}\label{lem:odd-unit}
Let $p$ be an odd prime and $\beta\ge1$.  Then
$$
U_{p,\beta}(T)
\equiv
(1-T^{p-1})^{p^{\beta-1}}
\pmod{p^\beta}.
$$
\end{lemma}

\begin{proof}
For $\beta=1$, the nonzero elements of $\mathbb F_p$ give
$$
\prod_{u=1}^{p-1}(X-u)=X^{p-1}-1.
$$
Since the set of nonzero residues is invariant under $u\mapsto-u$,
$$
U_{p,1}(T)\equiv1-T^{p-1}\pmod p.
$$

Now let $\beta\ge2$.  Every unit $u$ modulo $p^{\beta-1}$ has the
$p$ lifts $u+kp^{\beta-1}$ with $0\le k\le p-1.$
Let $A=1+uT, B=p^{\beta-1}T.$ Then
$$
\prod_{k=0}^{p-1}(A+kB)
\equiv A^p\pmod{p^\beta}.
$$
Indeed, the term linear in $B$ contains $\sum_{k=0}^{p-1}k=\frac{p(p-1)}2,$ and is therefore divisible by $p^\beta$, while every term containing
$B^r$ ($r\ge2$) is divisible by $p^{r(\beta-1)}$, hence by
$p^\beta$.  Multiplication over all units modulo $p^{\beta-1}$ yields
$$
U_{p,\beta}(T)\equiv U_{p,\beta-1}(T)^p\pmod{p^\beta}.
$$

If $U_{p,\beta-1}(T)\equiv V(T)\pmod{p^{\beta-1}},$ then $U_{p,\beta-1}(T)^p\equiv V(T)^p\pmod{p^\beta}.$ Now, we have
$$
U_{p,\beta}(T)
\equiv
(1-T^{p-1})^{p^{\beta-1}}
\pmod{p^\beta}.
$$
\end{proof}

For $1\le e\le\beta$, let $L_e:=\varphi(p^e)=p^{e-1}(p-1).$

\begin{lemma}\label{lem:odd-support}
Let $p$ be odd prime and $1\le e\le\beta$. If $L_e\nmid r$, then
$$
\nu_p\bigl([T^r]H_{p,\beta}(T)\bigr)
\ge\beta-e+1.
$$
Equivalently,
$$
H_{p,\beta}(T)
\in
\bigl(\Z/p^{\beta-e+1}\Z\bigr)[T^{L_e}].
$$
\end{lemma}

\begin{proof}
By Lemma \ref{lem:odd-unit},
$$
U_{p,\beta}(T)\equiv \sum_{m=0}^{p^{\beta-1}} (-1)^m\binom{p^{\beta-1}}mT^{m(p-1)} \pmod{p^\beta}.
$$
If $L_e\nmid m(p-1)$, then $p^{e-1}\nmid m$, so
$\nu_p(m)\le e-2$.  Lemma \ref{lem:binom-valuation} gives
$$
\nu_p\binom{p^{\beta-1}}m = \beta-1-\nu_p(m)\ge\beta-e+1.
$$
Thus
$$
U_{p,\beta}(T)\in\left(\Z/p^{\beta-e+1}\Z\right)[T^{L_e}].
$$

Now, we use the induction on $\beta$. If $\beta=e$, then
$H_{p,\beta-1}(pT)\equiv1\pmod p$, and the result follows from
\eqref{eq:H-rec}.  If $\beta>e$, the induction hypothesis says that
every coefficient of $H_{p,\beta-1}(T)$ in a degree not divisible by
$L_e$ is divisible by $p^{\beta-e}$.  After the substitution
$T\mapsto pT$, every positive degree gains at least one further factor
of $p$.  Hence
$$
H_{p,\beta-1}(pT)
\in
\bigl(\Z/p^{\beta-e+1}\Z\bigr)[T^{L_e}].
$$
Multiplying by $U_{p,\beta}(T)$ proves the claim.
\end{proof}

\begin{lemma}\label{lem:odd-threshold}
Let $p$ be odd prime and $1\le e\le\beta$.  Then
$$
\nu_p\bigl([T^{L_e}]H_{p,\beta}(T)\bigr)=\beta-e.
$$
\end{lemma}

\begin{proof}
We induct on $\beta$, with $e$ fixed.  From Lemma \ref{lem:odd-unit}, the coefficient $u_{L_e}$ of $T^{L_e}$
in $U_{p,\beta}$ satisfies
$$
u_{L_e}\equiv (-1)^{p^{e-1}}\binom{p^{\beta-1}}{p^{e-1}}\pmod{p^\beta}.
$$
By Lemma \ref{lem:binom-valuation}, $\nu_p(u_{L_e})=\beta-e.$

Let us write $H_{p,\beta-1}(T)=\sum_{r\ge0}c_rT^r.$ From \eqref{eq:H-rec}, we have
$$
[T^{L_e}]H_{p,\beta} = u_{L_e} + \sum_{r=1}^{L_e}u_{L_e-r}c_rp^r.
$$
For $0<r<L_e$, Lemma \ref{lem:odd-support} applied to $H_{p,\beta-1}$ gives 
$$
\nu_p(c_r)\ge\beta-e,
$$
and therefore
$$
\nu_p(c_rp^r)\ge\beta-e+1.
$$
For $r=L_e$, the induction hypothesis gives
$$
\nu_p(c_{L_e})=(\beta-1)-e, 
$$
whenever $\beta>e$, and hence
$$
\nu_p(c_{L_e}p^{L_e})
=
\beta-1-e+L_e
\ge\beta-e+1,
$$
since $L_e\ge2$.  In the base case $\beta=e$, the same term is
automatically divisible by $p$, which is already one order above the
valuation $0$ of $u_{L_e}$.

Thus every non-leading convolution term is divisible by
$p^{\beta-e+1}$, whereas $u_{L_e}$ has valuation exactly
$\beta-e$.  No cancellation can raise the valuation, proving the
claim.
\end{proof}

For $\beta\ge2$, let
$$
U_{2,\beta}(T)
=
\prod_{\substack{1\le r<2^\beta\\r\text{ odd}}}(1+rT).
$$

\begin{lemma}\label{lem:two-unit}
For every $\beta\ge2$,
$$
U_{2,\beta}(T)
\equiv
(1-T^2)^{2^{\beta-2}}
\pmod{2^\beta}.
$$
\end{lemma}

\begin{proof}
For $\beta=2$, we have
$$
U_{2,2}(T)=(1+T)(1+3T)
=1+4T+3T^2
\equiv1-T^2\pmod4.
$$

Let $\beta\ge3$.  Every odd residue $u$ modulo $2^{\beta-1}$
has the two lifts $u$ and $u+c$ ($c=2^{\beta-1}$).  Let
$A_u=1+uT$.  Then, 
$$
U_{2,\beta}(T)
=
\prod_u A_u(A_u+cT).
$$
Note that terms involving two or more factors $cT$ vanish modulo $2$.
Thus
$$
U_{2,\beta}(T)
\equiv
U_{2,\beta-1}(T)^2
+
cT\sum_u A_u\prod_{v\ne u}A_v^2
\pmod{2^\beta}.
$$
Note that $A_u$ is congruent to $1+T$ modulo $2$, and the number of odd residues
$u$ modulo $2^{\beta-1}$ is $2^{\beta-2}$, which is even.
Hence, the displayed sum is even.  The correction term is therefore
divisible by $2^\beta$, and
$$
U_{2,\beta}(T)\equiv U_{2,\beta-1}(T)^2\pmod{2^\beta}.
$$
Then, we use the following induction. Let us assume that
$$
U_{2,\beta-1}(T) \equiv B(T):= (1-T^2)^{2^{\beta-3}}\pmod{2^{\beta-1}}.
$$
Then $U_{2,\beta-1} - B(T)$ is divisible by $2^{\beta-1}$, while $U_{2,\beta-1} + B(T)$ is divisible by $2$. Hence
$$
U_{2,\beta-1}^2 - B(T)^2 \equiv 0\pmod{2^{\beta}}.
$$
Combining this and $U_{2,\beta}(T) \equiv U_{2,\beta-1}(T)\pmod{2^{\beta}}$ yields
$$
U_{2,\beta}(T)\equiv (1-T^2)^{2^{\beta-2}}\pmod{2^{\beta}}.
$$
\end{proof}

\begin{lemma}\label{lem:two-support}
Let $\beta\ge2$ and $1\le e\le\beta-1$.  If $2^e\nmid r$, then
$$
\nu_2\bigl([T^r]H_{2,\beta}(T)\bigr)\ge\beta-e.
$$
Equivalently,
$$
H_{2,\beta}(T)\in\left(\Z/2^{\beta-e}\Z\right)[T^{2^e}].
$$
\end{lemma}

\begin{proof}
By Lemma \ref{lem:two-unit},
$$
U_{2,\beta}(T)\equiv \sum_m(-1)^m \binom{2^{\beta-2}}mT^{2m} \pmod{2^\beta}.
$$
If $2^e\nmid 2m$, then $2^{e-1}\nmid m$, so $\nu_2(m)\le e-2$.  Hence 
$$
\nu_2\left(\binom{2^{\beta-2}}m\right)= \beta-2-\nu_2(m) \ge\beta-e.
$$
Therefore
$$
U_{2,\beta}(T)\in \left(\Z/2^{\beta-e}\Z\right)[T^{2^e}].
$$

Using $H_{2,\beta}(T) = U_{2,\beta}(T)H_{2,\beta-1}(2T),$ the proof now proceeds exactly as in the odd case.  If
$e=\beta-1$, the required modulus is $2$ and $H_{2,\beta-1}(2T)\equiv1\pmod2$.  Otherwise, the induction hypothesis,
followed by $T\mapsto2T$, raises every forbidden positive-degree coefficient by one further factor of $2$.  This proves the claim.
\end{proof}

\begin{lemma}\label{lem:two-threshold}
Let $\beta\ge2$ and $1\le e\le\beta-1$.  Then
$$
\nu_2\left([T^{2^e}]H_{2,\beta}(T)\right)
=
\beta-e-1.
$$
\end{lemma}

\begin{proof}
The coefficient $u_{2^e}$ of $T^{2^e}$ in $U_{2,\beta}$ has
$$
\nu_2(u_{2^e})
=
\nu_2\left(\binom{2^{\beta-2}}{2^{e-1}}\right)
=
\beta-e-1.
$$
Write
$$
H_{2,\beta-1}(T)=\sum_{r\ge0}c_rT^r.
$$
Then
$$
[T^{2^e}]H_{2,\beta}
=
u_{2^e}
+
\sum_{r=1}^{2^e}u_{2^e-r}c_r2^r.
$$
For $0<r<2^e$, Lemma~\ref{lem:two-support} gives
$$
\nu_2(c_r)\ge\beta-e-1,
$$
so
$$
\nu_2(c_r2^r)\ge\beta-e.
$$
For $r=2^e$, induction gives
$$
\nu_2(c_{2^e})=\beta-e-2
$$
when $\beta>e+1$, whence
$$
\nu_2(c_{2^e}2^{2^e})
=
\beta-e-2+2^e
\ge\beta-e.
$$
The boundary case $\beta=e+1$ is immediate modulo $2$.
Thus, all other convolution terms have valuation at least
$\beta-e$, one larger than the valuation of $u_{2^e}$, and the
result follows.
\end{proof}

We next transfer the preceding information to
$$
P_n(T)=\prod_{r=1}^{n-1}(1+rT).
$$

\begin{lemma}\label{lem:block-transfer}
Let $p$ be odd and let
$$
\alpha=\nu_p\left(\binom n2\right).
$$
If $\alpha>0$, then exactly one of $n$ or $n-1$ has $p$-adic valuation
$\alpha$.  Write that integer as
$$
N=Mp^\alpha,\qquad p\nmid M.
$$
Then
$$
P_n(T)\equiv H_{p,\alpha}(T)^M\pmod{p^\alpha}.
$$

For $p=2$, if
$$
\alpha=\nu_2\left(\binom n2\right)>0,
$$
whichever is even, $n$ or $n-1$ has valuation
$$
\beta=\alpha+1.
$$
Writing it as $N=M2^\beta$, $M$ odd, one has
$$
P_n(T)\equiv H_{2,\beta}(T)^M\pmod{2^\beta}.
$$
\end{lemma}

\begin{proof}
Partition the factors indexed by $1,\ldots,N$ into complete blocks of
length $p^\alpha$ in the odd case, or $2^\beta$ in the 2-adic case.
A shifted block is coefficientwise congruent to the initial block modulo
the corresponding prime power, while its final multiple of the block
length contributes a factor congruent to $1$.  If $N=n$, the extra
factor $1+nT$ is likewise congruent to $1$; if $N=n-1$, no extra
factor is needed.  This gives the stated congruences.
\end{proof}

For an odd prime $p$, let us define
$$
\mu_p(j):=
\max\{e\ge1:\varphi(p^e)\le2j\},
$$
with $\mu_p(j)=0$ if the set is empty.  For $p=2$, define
$$
\mu_2(j):=
\max\{e\ge1:2^{e-1}\le j\}.
$$

\begin{proposition}\label{prop:Mj-valuations}
For every prime $p$, $\nu_p(M_j)=\mu_p(j).$
\end{proposition}

\begin{proof}
For odd $p$, a number divisible by $p^e$ has an Euler function divisible
by $\varphi(p^e)$.  Hence, a $p^e$-factor can occur in
$\lcm\{m:\varphi(m)\le2j\}$ if and only if
$\varphi(p^e)\le2j$.  For $p=2$, the outer factor $1/2$ removes
exactly one power of $2$, giving the stated indexing.
\end{proof}

\begin{theorem}\label{thm:primewise-gcd}
Let
$$
\alpha=\nu_p(a).
$$
Then
$$
\min_{1\le i\le j}\nu_p(b_i) = \max\{0,\alpha-\mu_p(j)\}.
$$
\end{theorem}

\begin{proof}
Let us assume that $p$ is odd.  If $1\le e:=\mu_p(j)<\alpha$, then
$$
L_{e+1}=\varphi(p^{e+1})>2j.
$$
By Lemma \ref{lem:odd-support} and Lemma \ref{lem:block-transfer}, we have $\nu_p(b_i)\ge\alpha-e$ for every $1\le i\le j$.
On the other hand,
$$
h_e:=\frac{\varphi(p^e)}2\le j,
$$
and Lemmas~\ref{lem:odd-threshold} and
\ref{lem:block-transfer} give
$$
\nu_p(b_{h_e})=\alpha-e.
$$
If $e\ge\alpha$, the threshold $h_\alpha\le j$ gives a coefficient
prime to $p$, so the minimum is $0$.

If $\mu_p(j)=0$, then $p-1>2j$.  Block support gives
$\nu_p(b_i)\ge\alpha$.  Since $p\ge5$, $b_1 = \frac{n(n-1)(n-2)(3n-1)}{24}$ has valuation exactly $\alpha$: either $n$ or $n-1$ contains
$p^\alpha$, all the remaining displayed factors are $p$-adic units.
Thus, the formula holds.

Now, let $p=2$.  If $\alpha=0$, the assertion is
immediate, so assume $\alpha\ge1$ and put $\beta=\alpha+1$.
If $e:=\mu_2(j)<\alpha,$ then by the definition of $\mu_2$,
$$
2^{e-1}\le j<2^e.
$$
Apply Lemma~\ref{lem:two-support} at level $e+1$.  Since $2i\le2j<2^{e+1},$ we have $2^{e+1}\nmid2i$ for every $1\le i\le j$.  The block
transfer therefore gives
$$
\nu_2(b_i)\ge
\beta-(e+1)
=
\alpha-e
\qquad(1\le i\le j).
$$
On the other hand, the threshold degree $2^e$ corresponds to the
index $h_e=2^{e-1}\le j,$ and Lemma \ref{lem:two-threshold}, followed by the block transfer, gives
$$
\nu_2(b_{h_e})=\beta-e-1=\alpha-e.
$$
If $e\ge\alpha$, then $2^{\alpha-1}\le j$, and the threshold at
level $\alpha$ gives
$$
\nu_2(b_{2^{\alpha-1}})=0.
$$
This proves the formula.
\end{proof}

Now, we can prove the main results in Theorem \ref{thm:main} and Corollary \ref{cor:numerical} related to the exact gcd filtration.
\begin{proof}[Proof of Theorem~\ref{thm:main}]
For every prime $p$, Theorem \ref{thm:primewise-gcd} gives
$$
\nu_p(d_j)
=
\max\{0,\nu_p(a)-\mu_p(j)\}.
$$
By Proposition \ref{prop:Mj-valuations},
$$
\nu_p(\gcd(a,M_j))
=
\min\{\nu_p(a),\mu_p(j)\}.
$$
Hence, 
$$
\nu_p(d_j)
=
\nu_p\left(\frac{a}{\gcd(a,M_j)}\right)
$$
for every prime $p$.  The first assertion follows, and the formula for
$q_j=d_{j-1}/d_j$ is immediate.
\end{proof}
\begin{proof}[Proof of Corollary~\ref{cor:numerical}]
It remains to show that $d_s=1$.  By Theorem~\ref{thm:main},
$$
d_s=\frac{a}{\gcd(a,M_s)},
$$
so it suffices to prove that
$$
a\mid M_s.
$$

Let $p^\alpha\Vert a$.  Suppose first that $p$ is odd.  Since
$$
p^\alpha\mid \frac{n(n-1)}2,
$$
we have $p^\alpha\mid n$ or $p^\alpha\mid n-1$.

If $n$ is odd, then $2s=n-1$, and
$$
\varphi(p^\alpha)<p^\alpha\le n.
$$
Since both sides are integers,
$$
\varphi(p^\alpha)\le n-1=2s.
$$
If $n$ is even and $p^\alpha\mid n-1$, then similarly
$$
\varphi(p^\alpha)<p^\alpha\le n-1,
$$
and hence
$$
\varphi(p^\alpha)\le n-2=2s.
$$
Finally, if $n$ is even and $p^\alpha\mid n$, then, since
$p^\alpha$ is odd,
$$
p^\alpha\le\frac n2,
$$
and therefore
$$
\varphi(p^\alpha)<p^\alpha
\le\frac n2\le n-2=2s.
$$
Thus in every case $\mu_p(s)\ge\alpha$.  By
Proposition~\ref{prop:Mj-valuations},
$$
\nu_p(M_s)\ge\alpha.
$$

Now let $p=2$ and $\alpha=\nu_2(a)>0$.  The even member of
$\{n,n-1\}$ is divisible by $2^{\alpha+1}$.  If $n$ is odd, then
$$
2^{\alpha-1}\le\frac{n-1}{4}\le s,
$$
whereas if $n$ is even, then
$$
2^{\alpha-1}\le\frac n4\le\frac{n-2}{2}=s.
$$
Hence Proposition~\ref{prop:Mj-valuations} again gives
$$
\nu_2(M_s)\ge\alpha.
$$

We have therefore proved that $a\mid M_s$.  Hence
$$
d_s=\frac{a}{\gcd(a,M_s)}=1.
$$
By Corollary~\ref{cor:numerical}, $S_n$ is a numerical semigroup.
\end{proof}

\begin{corollary}\label{cor:q-squarefree}
Every $q_j$ is squarefree.
\end{corollary}

\begin{proof}
By Proposition~\ref{prop:Mj-valuations}, $\nu_p(M_j)=\mu_p(j)$.
For odd $p$, the successive threshold values
\[
\varphi(p^e)=p^{e-1}(p-1)
\]
are multiplied by $p\ge3$ as $e$ increases, while for $p=2$ the
successive thresholds are the powers $2^{e-1}$.  Hence, when the
integer index increases from $j-1$ to $j$, at most one new threshold
for a fixed prime $p$ can be crossed.  Thus
\[
\nu_p(M_j)-\nu_p(M_{j-1})\in\{0,1\}.
\]
Since
\[
q_j=\frac{\gcd(a,M_j)}{\gcd(a,M_{j-1})},
\]
it follows that $\nu_p(q_j)\in\{0,1\}$ for every prime $p$.
\end{proof}

\begin{lemma}\label{lem:support}
Let $p\mid q_h$.  Then, for every admissible $m$, $h\nmid m$ implies that $\nu_p(b_m)\ge \nu_p(d_h)+1.$
\end{lemma}

\begin{proof}
Suppose first that $p$ is odd.  Since $p\mid q_h$, the gcd
filtration shows that for some $e\ge1$,
$$
2h=\varphi(p^e),
\qquad
\nu_p(d_h)=\alpha-e,
\quad
\alpha=\nu_p(a).
$$
If $h\nmid m$, then
$$
\varphi(p^e)\nmid2m.
$$
The global block-support congruence obtained from
Lemmas~\ref{lem:odd-support} and \ref{lem:block-transfer} gives
$$
\nu_p(b_m)\ge\alpha-e+1=\nu_p(d_h)+1.
$$

For $p=2$, $2\mid q_h$ implies
$$
h=2^{e-1},
\qquad
\nu_2(d_h)=\alpha-e.
$$
If $h\nmid m$, then $2^e\nmid2m$.  Lemmas~\ref{lem:two-support} and
\ref{lem:block-transfer} give
$$
\nu_2(b_m)\ge\alpha-e+1=\nu_2(d_h)+1.
$$
\end{proof}

\begin{proposition}\label{prop:canonical}
For every $1\leq j\leq s$, there exist unique digits
$$
0\le r_{j,h}<q_h
$$
where $1\leq h < j$ and an integer $\lambda_j$ such that
$$
q_jb_j = \lambda_ja+\sum_{h=1}^{j-1}r_{j,h}b_h.
$$
Moreover, $r_{j,h}\ne0$ implies that $h\mid j$.
\end{proposition}

\begin{proof}
Since $d_j\mid b_j$ and $d_{j-1}=q_jd_j, d_{j-1}\mid q_jb_j.$ Let $N_{j-1}=q_jb_j$.  For $h=j-1,\ldots,1$, suppose that $d_h\mid N_h$.
Because $d_h=\gcd(d_{h-1},b_h)$ and $d_{h-1}=q_hd_h,$ we have
$$
\gcd\left(q_h,\frac{b_h}{d_h}\right)=1.
$$
Thus, there is a unique $0\le r_{j,h}<q_h$ such that
$$
N_{h-1}:=N_h-r_{j,h}b_h
$$
is divisible by $d_{h-1}$.  At the end, $a=d_0\mid N_0$, so
$N_0=\lambda_ja$.

It remains to prove divisor support.  We use descending induction on
$h$.  Assume $h\nmid j$.  If $k>h$ and $r_{j,k}\ne0$, the
induction hypothesis gives $k\mid j$; hence $h\nmid k$, for otherwise
$h\mid k\mid j$.

Let $p\mid q_h$.  By Lemma~\ref{lem:support},
$$
\nu_p(b_j),\,\nu_p(b_k)\ge \nu_p(d_h)+1
$$
for every surviving $k>h$.  Hence $p\mid\frac{N_h}{d_h}$.  Moreover, by construction,
\[
N_{h-1}=N_h-r_{j,h}b_h
\]
is divisible by $d_{h-1}=q_hd_h$.  Dividing by $d_h$ therefore gives
\[
\frac{N_h}{d_h}
\equiv
r_{j,h}\frac{b_h}{d_h}
\pmod{q_h}.
\]
The factor $b_h/d_h$ is a $p$-adic unit, so $p\mid r_{j,h}$.
This holds for every $p\mid q_h$.  Since $q_h$ is squarefree,
$q_h\mid r_{j,h}$, and the digit range forces $r_{j,h}=0$.
\end{proof}

\begin{lemma}\label{lem:extension}
Let $0<x_1\le\cdots\le x_N$ and let $r,s\ge1$ with $r+s\le N$.  Then
$$
\frac{e_{r+s}(x_1,\ldots,x_N)}
     {e_r(x_1,\ldots,x_N)}
\ge
\frac{e_s(x_1,\ldots,x_{N-r})}
     {\binom{r+s}{r}}.
$$
If the $x_i$ are strictly increasing, the inequality is strict.
\end{lemma}

\begin{proof}
For $B\subseteq[N]$, write $x_B=\prod_{i\in B}x_i$.  Double counting
the monomials of degree $r+s$ gives
$$
\sum_{\substack{B\subseteq[N]\\|B|=r}}
x_B\,e_s(x_i:i\notin B)
=
\binom{r+s}{r}e_{r+s}(x_1,\ldots,x_N).
$$
Dividing by $e_r(x_1,\ldots,x_N)=\sum_{|B|=r}x_B$ shows that the left-hand side, after normalization, is a weighted average
of the values $e_s$. Note that $x_i$ is chosen under the condition $i\notin B$.

If $[N]\setminus B=\{i_1<\cdots<i_{N-r}\},$ then $i_t\ge t$, and hence $x_{i_t}\ge x_t$.  Since $e_s$ is
increasing in every positive variable,
$$
e_s(x_i:i\notin B)\ge e_s(x_1,\ldots,x_{N-r}).
$$
The asserted inequality follows.  Strictness follows when the $x_i$
are strictly increasing because every $r$-subset $B$ does not leave the
same minimal complement.
\end{proof}

\begin{corollary}\label{cor:stirling-growth}
For $1\le h<j\le s$,
$$
\frac{b_j}{b_h} > \frac{e_{2(j-h)}(1,\ldots,n-1-2h)}{\binom{2j}{2h}}\ge \frac{(2j-2h)!}{\binom{2j}{2h}}.
$$
Consequently,
$$
\frac{b_h}{b_j} < \frac{\binom{2j}{2h}}{(2j-2h)!}.
$$
\end{corollary}

\begin{proof}
Apply Lemma~\ref{lem:extension} with
$$
N=n-1,\qquad r=2h,\qquad s=2(j-h),\qquad x_i=i.
$$
The second inequality follows by retaining the single positive monomial
$1\cdot2\cdots(2j-2h)=(2j-2h)!$.
\end{proof}

Let $C_h:=\frac{M_h}{M_{h-1}}.$ Then $q_h\mid C_h$, hence $q_h\le C_h$.

\begin{lemma}\label{lem:C-bound}
For every $h\ge1$,
$$
C_h\le2(2h+1)(\sqrt{2h}+1).
$$
\end{lemma}

\begin{proof}
At a fixed threshold $h$, a new odd prime with exponent $1$ can only
be $2h+1$.  A higher odd prime-power contribution satisfies $p^{e-1}(p-1)=2h$ for any $e\ge2,$ and hence
$$
p(p-1)\le2h,
$$
with $p < \sqrt{2h}+1.$ 

There is at most one such odd prime $p$.  Indeed, if distinct odd
primes $p<q$ with exponents $e,f\ge2$ satisfied
$$
p^{e-1}(p-1)=q^{f-1}(q-1),
$$
then $q$ would divide $p^{e-1}(p-1)$.  Since $q\ne p$, this would
force $q\mid p-1$, impossible because $q>p-1$.
Finally, the $2$-adic exponent can increase by at most one at a fixed
threshold.  Thus $C_h$ is bounded by the product of at most the three
factors $2, 2h+1, \sqrt{2h}+1,$ which proves the claim.
\end{proof}

For later finite verification, the first universal jumps are
$$
(C_1,\ldots,C_{14})
=
(6,10,21,2,11,13,1,34,57,5,23,1,1,29).
$$

\begin{theorem}\label{thm:positivity}
For every admissible $j$, the integer $\lambda_j$ in
Proposition~\ref{prop:canonical} is positive.  Consequently,
$$
q_jb_j\in\langle a,b_1,\ldots,b_{j-1}\rangle.
$$
\end{theorem}

\begin{proof}
By Proposition~\ref{prop:canonical},
$$
q_jb_j=\lambda_ja+A_j,
$$
with
$$
A_j:=
\sum_{\substack{h\mid j\\h<j}}r_{j,h}b_h.
$$
Since $0\le r_{j,h}<q_h$,
$$
A_j\le
\sum_{\substack{h\mid j\\h<j}}(q_h-1)b_h.
$$
It suffices to prove $A_j<q_jb_j$; in most cases we prove the stronger
inequality $A_j<b_j$.

For $5\le j\le14$, $j\ne6$, Corollary~\ref{cor:stirling-growth}
gives
$$
\frac{A_j}{b_j}
<
E_j:=
\sum_{\substack{h\mid j\\h<j}}
(C_h-1)
\frac{\binom{2j}{2h}}{(2j-2h)!}.
$$
Direct substitution of the universal jumps $C_h$ for the proper divisors $h\mid j$ gives the following exact rational values:
$$
\begin{array}{c|c}
j&E_j\\ \hline
5&5/896\\
7&13/13685760\\
8&57979127/181621440\\
9&360389137/464950886400\\
10&40746147528721/80029671321600\\
11&1/2106408665088000\\
12&93315263272827703/1377451872276480000\\
13&1/381814401066608885760\\
14&233/784491637734604800000
\end{array}
$$
and every entry is $<1$.  Hence $A_j<b_j$.

Now let $j\ge15$.  Every proper divisor $h\mid j$ satisfies
$h\le j/2$, and Corollary~\ref{cor:stirling-growth} yields the crude
uniform bound
$$
\frac{b_h}{b_j}<\frac{4^j}{j!}.
$$
By Lemma~\ref{lem:C-bound},
$$
q_h\le2(j+1)(\sqrt j+1)
$$
with $h\le j/2$. There are at most $j/2$ proper positive divisors, so
$$
\frac{A_j}{b_j}
<
G(j):=
\frac{j(j+1)(\sqrt j+1)4^j}{j!}.
$$
Direct calculation yields,
$$
G(15)=0.960299\ldots<1,
$$
and
$$
\frac{G(j+1)}{G(j)}
=
\frac{4(j+2)}{j(j+1)}
\frac{\sqrt{j+1}+1}{\sqrt j+1}
<1
\qquad(j\ge15).
$$
Thus $A_j<b_j$ for every $j\ge15$.

Now, we treat the remaining cases: $j=1,2,3,4,6$.

If $j=1$, there is no remainder term, and $q_1b_1=\lambda_1a$ with $\lambda_1>0$.

For $j=2$, $A_2\le5b_1.$
If $n\ge7$, then we have
$$
\frac{b_2}{b_1}
>
\frac{e_2(1,\ldots,n-3)}6
\ge\frac{e_2(1,2,3,4)}6
=\frac{35}{6}>5.
$$
For the remaining cases, 
$$
n=5:\quad 5b_2=120=12a,\qquad\text{and}\qquad n=6:\quad 5b_2=1370=80a+2b_1.
$$

For $j=3$, $A_3\le5b_1.$ If $n\ge8$, then we have
$$
\frac{b_3}{b_1} > \frac{e_4(1,\ldots,n-3)}{15}\ge\frac{274}{15}>5.
$$
For $n=7, 7b_3=5040=240a.$

For $j=4$, $A_4\le5b_1+9b_2$.  Although $b_3$ does not occur in
$A_4$ (since $3\nmid4$), we use it as an intermediate comparison
term.  For $n\ge11$, one has $b_2>5b_1$, and
$$
\frac{b_3}{b_2}
>
\frac{e_2(1,\ldots,6)}{15}
=
\frac{175}{15}>10,
$$
so
$$
b_3>9b_2+5b_1.
$$
Furthermore,
$$
\frac{b_4}{b_3}
>
\frac{e_2(1,2,3,4)}{28}
=
\frac{35}{28}>1.
$$
Hence $b_4>A_4$.  For $n=10$,
$$
b_4=1026576>
5(870)+9(63273)=573807,
$$
and for $n=9$,
$$
b_4=40320=1120a.
$$

Finally let $j=6$.  Then
$$
A_6\le5b_1+9b_2+20b_3.
$$
For $n\ge15$, the preceding estimates give
$$
b_3>5b_1+9b_2,
$$
and hence $A_6<21b_3$.  Also
$$
\frac{b_6}{b_3}
>
\frac{e_6(1,\ldots,n-7)}{\binom{12}{6}}
>
\frac{3\cdot4\cdot5\cdot6\cdot7\cdot8}{924}
=
\frac{240}{11}>21.
$$
Thus $A_6<b_6$.

For $n=13$, $q_6=13, b_6=12!,$ and
$$
13b_6=13!=6227020800 = 79833600a.
$$
For $n=14$,
$$
(q_1,\ldots,q_6)=(1,1,7,1,1,13),
$$
so $A_6\le6b_3$, while
$$
\frac{q_6b_6}{b_3}
>
13\frac{2\cdot3\cdot4\cdot5\cdot6\cdot7}{924}
=
13\frac{60}{11}>6.
$$
Hence $q_6b_6>A_6$.

Thus $\lambda_j>0$ in every case.
\end{proof}

\begin{proof}[Proof of Theorem~\ref{thm:telescopic}]
Let
$$
j_1<j_2<\cdots<j_t
$$
be precisely the indices for which $q_{j_r}>1$.
If $q_j=1$, Theorem~\ref{thm:positivity} gives
$$
b_j\in\langle a,b_1,\ldots,b_{j-1}\rangle,
$$
so inactive generators can be removed.  Hence
$$
S_n=\langle a,b_{j_1},\ldots,b_{j_t}\rangle.
$$

For every active index, Theorem~\ref{thm:positivity} gives
$$
q_{j_r}b_{j_r}
\in
\langle a,b_{j_1},\ldots,b_{j_{r-1}}\rangle.
$$
The gcd quotients of this reduced sequence are
$q_{j_1},\ldots,q_{j_t}$, all greater than $1$.
Therefore the reduced sequence is telescopic in the standard sense;
see \cite{Rosales2009,Shor2019}.
\end{proof}
\begin{proof}[Proof of Corollary~\ref{cor:apery-frobenius}]
The standard rectangular Ap\'ery-set description for a telescopic
generating sequence (see, e.g.,
\cite{Rosales2009,Shor2019}) gives
$$
\Ap(S_n,a)
=
\left\{
\sum_{r=1}^{t}x_rb_{j_r}:
0\le x_r<q_{j_r}
\right\}.
$$
Since $q_j=1$ forces $x_j=0$, this is equivalently
$$
\Ap(S_n,a)
=
\left\{
\sum_{j=1}^{s}x_jb_j:
0\le x_j<q_j
\right\}.
$$
Its maximum is
$$
\max\Ap(S_n,a)
=
\sum_{j=1}^{s}(q_j-1)b_j.
$$
By the standard relation between the Frobenius number and
the Ap\'ery set \cite{Rosales2009},
$$
F(S_n)=\max\Ap(S_n,a)-a.
$$
Therefore,
$$
F(S_n)
=
\sum_{j=1}^{s}(q_j-1)b_j-a,
$$
and hence
$$
F(S_n)=
\sum_{j=1}^{\lfloor(n-1)/2\rfloor}
(q_j-1)\stirling n{n-2j}
-\binom n2.
$$
\end{proof}
\begin{proof}[Proof of Corollary~\ref{cor:structural-invariants}]
By Theorem~\ref{thm:telescopic}, after deleting the inactive generators
with $q_j=1$, the resulting generating sequence is telescopic.  Hence
$S_n$ is free with respect to this telescopic arrangement; see
\cite{Rosales2009,Shor2019}.

We next derive symmetry directly from the rectangular Ap\'ery set.
Since $d_0=a$ and $d_s=1$,
\[
\prod_{j=1}^s q_j=\frac{d_0}{d_s}=a.
\]
By Corollary~\ref{cor:apery-frobenius}, each element of
$\Ap(S_n,a)$ is represented uniquely by digits $0\le x_j<q_j$.
Therefore, summing coordinatewise over the rectangular Ap\'ery set,
\[
\frac1a\sum_{w\in\Ap(S_n,a)}w
=\frac12\sum_{j=1}^s(q_j-1)b_j.
\]
The standard Ap\'ery-set formula for the genus gives
\[
\mathrm g(S_n)
=\frac1a\sum_{w\in\Ap(S_n,a)}w-\frac{a-1}{2}.
\]
Consequently,
\[
\mathrm g(S_n)
=\frac12\left(\sum_{j=1}^s(q_j-1)b_j-a+1\right)
=\frac{F(S_n)+1}{2}.
\]
The equality $\mathrm g(S_n)=(F(S_n)+1)/2$ characterizes symmetric
numerical semigroups, so $S_n$ is symmetric and hence has type one;
see \cite{Rosales2009}.  Finally, $c(S_n)=F(S_n)+1$ for every numerical
semigroup.  Substituting the formula for $F(S_n)$ from
Corollary~\ref{cor:apery-frobenius} gives the displayed formulas for
$\mathrm g(S_n)$ and $c(S_n)$.
\end{proof}

\section{Future work}   

As a generalized version of Frobenius numbers, $p$-Frobenius numbers are well-known (see, e.g., \cite{Ko-p}).  For example, when $n=7$, for $0\le p\le 9$, we have    
\begin{align*}
	g_p(S_7)&=(3 p+2)\stirling{7}{5}+6\stirling{7}{1}-\stirling{7}{6}\\
	&=(3 p+2)\cdot 175+6\cdot 720-21=4649+525 p. 
\end{align*}
However, it does not seem easy to formulate $p$-Frobenius numbers for a general $n$. Even in the case of binomial coefficients (\cite{HwangSong26}), it seems pretty difficult.

\end{document}